\documentclass[leqno]{amsart}

\usepackage{amsmath}
\usepackage{amsfonts}
\usepackage{amssymb}
\usepackage{pdfsync}
\usepackage{color,graphicx,tikz-cd}
\usepackage{a4wide}
\usepackage{amscd,amsthm,latexsym}
\usepackage{hyperref}
\usepackage{kotex}
\usepackage{relsize}
\usepackage{cite}

\allowdisplaybreaks %

\title{Exploring  a Delta Schur  Conjecture}

\author{Adriano Garsia}
\address{Department of Mathematics, University of California, San Diego, La Jolla, CA, USA}
\email{garsia@math.ucsd.edu}

\author{Jeffrey Liese}
\address{
Department of Mathematics, California Polytechnic State University, San Luis Obispo, CA, USA}
\email{jliese@calpoly.edu}

\author{Jeffrey B. Remmel}
\address{
Department of Mathematics, University of California, San Diego, La Jolla, CA, USA}
\email{jremmel@ucsd.edu}

\author{Meesue Yoo}
\address{
Applied Algebra and Optimization Research Center, Sungkyunkwan University, Suwon,
South Korea}
\email{meesue.yoo@skku.edu}

\date{\today}

\thanks{The first named author was supported by NSF grant DMS1700233.
 The fourth author was supported by NRF grants $2016R1A5A1008055$ and $2017R1C1B2005653$.
}



\date{\today}

\newtheorem{thm}{Theorem}[section]

\newtheorem{prop}[thm]{Proposition}
\newtheorem{cor}[thm]{Corollary}
\theoremstyle{definition}

\newtheorem{remark}[thm]{Remark}

\def\multi#1{\vbox{\baselineskip=0pt\halign{\hfil$\scriptstyle\vphantom{(_)}##$\hfil\cr#1\crcr}}}
\def \o1  {{\overline }}
\def \l {{\ell}}

\def \bu {\hskip -.15in}

\def\tttt #1{{\textstyle{#1} }}

\def \uu {\hskip -.08in}

\def \magstep#1 {\ifcase#1 1000\or 1200\or 1440\or 1728\or 2074\or 2488\fi\relax}

\def\la{{\lambda}}

\font\title=cmbx10 scaled\magstep2

\def \-> {\rightarrow}
\def\LL{\big\langle}

\def\RR {\big\rangle}

\def\DD {\Delta}

\def\om {\omega}
\def\la {\lambda}

\def \RA {\rightarrow}

\def \sas {\vskip .06truein}
\def\sa{{\vskip .125truein}}

\def \ses {\enskip = \enskip}
\def \sps {\, + \,}

\def \scs {\, , \,}
\def \ess {\enskip}

\def \ssp {\hskip .25em}
\def \bigsp {\hskip .5truein}
\def \part {\vdash}

\def \DD {\Delta}

\def \RA {{ \rightarrow }}

\def \om {\omega}

\def \TH {{\widetilde H}}

\def \om {\omega}

\def \TH {{\widetilde H}}

\def \scs {\ssp , \ssp}
\def \ess {\enskip}
\def \ssp {\hskip .25em}
\def \bigsp {\hskip .5truein}
\def \part {\vdash}

\newcommand{\qbin}[3]{\genfrac{[}{]}{0pt}{}{#1}{#2}_{#3}}

\begin{document}

\begin{abstract}
 In \cite{HRW15}, Haglund, Remmel, Wilson state a conjecture which predicts
a purely combinatorial way of obtaining the symmetric function $\Delta_{e_k}e_n$. It is  called the Delta Conjecture. It was recently proved in \cite{GHRY} that the Delta Conjecture is true when either $q=0$ or $t=0$.
In this paper we complete a work initiated by Remmel whose initial aim was to explore the symmetric function
$\Delta_{s_\nu} e_n$ by the same methods developed in
\cite{GHRY}. Our first need here is a method for constructing a symmetric function that may be viewed as a ``combinatorial side'' for  the symmetric function $\Delta_{s_\nu} e_n$ for $t=0$. Based on what was discovered in \cite{GHRY} we conjectured such a construction mechanism.  We prove here that in the case that  $\nu=(m-k,1^k)$ with $1\le m< n$ the  equality of the two sides can be established by the same methods used in \cite{GHRY}. While this work was in progress, we learned that Rhodes and Shimozono   had  previously constructed also such a ``combinatorial side''.
Very recently, Jim Haglund was able to prove that their conjecture follows from the results in \cite{GHRY}.
We show here that an  appropriate
modification of the Haglund arguments proves that the polynomial $\Delta_{s_\nu}e_n$
as well as the Rhoades-Shimozono ``combinatorial side''
have a plethystic evaluation with hook Schur function
expansion.
\end{abstract}


\maketitle

\vspace{-.3 in}
\section{Introduction}

Our manipulations rely heavily on plethystic notation and the terminology used in \cite{GHRY}. In \cite{HRW15}, the reader can find detailed explanations for all of the notations used in this paper.

Recall that Dyck paths in the $n\times n$ lattice square $R_n$ are   paths from $(0,0)$ to $(n,n)$ proceeding by north and east  unit steps, always remaining  weakly above the main diagonal of $R_n$.  These paths are usually represented by their area sequence
$(a_1,a_2,\ldots, a_n)$, where $a_i$ counts the number of complete cells
 between the north step in the $i^{th}$ row and the diagonal. Notice that the $x$-coordinate of the north step in the $i^{th}$ row is simply the difference $u_i=i-1-a_i$.

 A parking function $PF$ supported by the Dyck path $D\in R_n$ is obtained by labeling the north steps of $D$ with $1,2,\ldots ,n$ (usually referred as ``cars''), where the labels increase along the north segments of $D$. Parking functions can be represented as two line arrays
$$
PF=\begin{pmatrix}
c_1 & c_2  &  \cdots  &   c_n\\
a_1 & a_2  &  \cdots  &   a_n\\
\end{pmatrix}
$$
 with cars $c_i$ and area numbers $a_i$ listed from bottom to top.
We also set
$$
area(PF)= \sum_{i=1}^n a_i,
\ess\ess\ess
dinv(PF)= \uu \sum_{1\le i<j\le n}\uu \Big(
\chi(c_i<c_j \ess\&\ess a_i=a_j)\sps
\chi(c_i>c_j \ess\&\ess a_i=a_j+1)
\Big).
$$
Moreover, the word $w(PF)$ is the permutation obtained by reading the cars
in the two line array by decreasing area numbers and from right to left.

 This given, the Haglund factor of a Dyck path $D$ is obtained by setting
$$
H_D(z;t)\ses \prod_{i=2}^{n}\big(1+{z\over t^{a_i}}\big)^{\chi(u_{i-1}= u_i)}.
$$
The $LLT$ polynomial constructed
from the Dyck path $D$ is  obtained by setting
$$
LLT_D(X;q,t)\ses \sum_{D(PF)=D}t^{area(PF)}q^{dinv(PF)}s_{p\big(ides(w(PF)\big)}[X]
$$
where the sum is over parking functions supported by $D$ and the last factor is the Schur function indexed by the composition giving the descent set of the inverse of $w(PF)$.

The special version of the Delta Conjecture of \cite{HRW15} we refer to here is the equality
\begin{equation}\label{eq:theDC}
\DD'_{e_{k-1}} e_n\ses
\sum_{D\in R_n}
LLT_D(X;q,t)\ess
H_D(z;t)\, \Big|_{z^{n-k}}
\end{equation}
where $\DD'_{F}$ is the eigen-operator of the modified Macdonald polynomial
defined by setting for any symmetric function $F$
$$
\DD'_{F}\TH_\mu[X;q,t]\ses
F\big[B_\mu(q,t)-1\big]\, \TH_\mu[X;q,t]
\ess\ess\ess
(\hbox{for all $\mu$}).
$$
As mentioned previously, it was proved in \cite{GHRY} that
the equality in (\ref{eq:theDC}) is valid when both sides are evaluated at $q=0$. Since the left hand side is easily shown to be symmetric in $q$ and $t$, then it must also remain valid when both sides are evaluated at $t=0$.

The main result in \cite{GHRY} is the equality of the symmetric functions on the right hand sides of
the following two equations
\begin{equation}\label{eq:24fromGHRY}
\sum_{\la\part n}LHS_{k,\la}s_{\la'}[x(1-q)]=
\sum_{\mu\part n}q^{- n(\mu)}
P_\mu[X;1/q]
\Big[{ \ell(\mu)-1\atop k-1 }\Big]_q
(q;q)_{\ell(\mu)},
\end{equation}
and
\begin{equation}\label{eq:25fromGHRY}
\sum_{\la\part n}RHS_{k,\la}s_{\la'}[X(1-q)]\ses
q^{-k(k-1)}
(q;q)_k\sum_{\substack{\mu\part n\\ \ell(\mu)=k}}
q^{n(\mu)}
 P_\mu[X;q],
\end{equation}
(these are labeled   (24) and (25) in that paper).
Where
\[
LHS_{k,\la}=q^{-\binom{k}{2}}\langle\Delta_{e_{k-1}} ' e_n, s_\lambda\rangle\Big|_{\substack{t=0}}
,
\ess\ess\ess
RHS_{k,\la}= \sum_{D\in R_n}
\LL LLT_D(X;q,t) ,s_\la\RR
H_D(z;t) \Big|_{z^{n-k}}\Big|_{t=0},
\]
and $ P_\mu[X;q]$, $ Q_\mu[X;q]$ are the Hall-Littlewood polynomials with Cauchy Kernel
\[
\sum_{\mu\part n}P_\mu[X;q]\, Q_\mu[Y;q]
\,=\, h_n\big[XY(1-q)\big]
\]

The present work was started by Jeff Remmel who sadly passed away before its completion. Remmel proposed the possibility of extending the Delta Conjecture when
the symmetric function side
``$\DD'_{e_k}e_n$'' is replaced by
``$\DD'_{s_\nu}e_n$'' , with $\nu$ an arbitrary partition. Remmel asked
the first author to obtain computer data to see if there was any similarity to the data that was obtained in the classical case.  One of the most surprising features of the classical case
is the discovery that the polynomial in (\ref{eq:25fromGHRY}) contains only hook Schur functions in its Schur expansion.
It is precisely this experimental discovery that made the proof of the equality of the polynomials in (\ref{eq:24fromGHRY}) and (\ref{eq:25fromGHRY}) substantially less challenging.

This given, we began an exploration of the Schur expansion of the polynomial
\begin{equation}\label{eq:SymSide}
LHS_{\nu,n}[X;q]\ses\sum_{\la\part n}\LL\DD'_{s_\nu}e_n\scs
s_\la\RR\Big|_{t=0} s_{\la'}[X(1-q)].
\end{equation}
To our surprise, this polynomial also yielded Schur expansions containing only hook Schur functions.
\sas

A crucial feature of \cite{GHRY} was the discovery of a new method for proving the equality of two symmetric functions.
More precisely, the equality of the functions in (\ref{eq:24fromGHRY}) and (\ref{eq:25fromGHRY}) as well as their hook Schur function expansion
was obtained simply by showing that both could be expressed as linear combinations of the following shifted Cauchy kernel, using  the same coefficients $c_i(q)$
\[
\sum_{\mu\part n}P_\mu[X;q]\, Q_\mu[\tttt{ 1-q^i \over  1-q };q]
\,=\,h_n\big[X(1-q^i)\big]
\bigsp (\hbox{for $1\le i\le n$})
\]

The data obtained, in the present case, suggested that all these desired features are present only when
$\nu$ is restricted to be a hook partition $(m-k,1^k)$  with $m<n$.
This discovery prompted us to study the symmetric function
\begin{equation}\label{eq:polyToStudy}
{\it LHS}_{k,m,n}[X,q]\ses \om\Big(\DD_{s_{m-k,1^k}}'e_n\Big|_{t=0}
\Big) [X(1-q)].
\end{equation}
Following the basic steps carried out in \cite{GHRY}  we prove here  that (\ref{eq:polyToStudy}) is equivalent to the identity
\begin{equation}\label{eq:equivalentPoly}
{\it LHS}_{k,m,n}[X,q]= q^{m+{k+1\choose2}}
\sum_{\mu\part  n}   q^{-n(\mu)}(q;q)_{\ell (\mu)}
\Big[{  m -1\atop k}\Big]_q
\Big[{m +\ell(\mu) -(k+2)\atop m}\Big]_q  P_\mu [X;q^{-1}].
\end{equation}

 To mimic the methods used in the classical case, we now need   to produce a ``combinatorial side''.  A simple comparison of the right hand sides of (\ref{eq:24fromGHRY}) and (\ref{eq:25fromGHRY}) shows that, in the case of the Delta Conjecture,
  the symmetric function produced by the ``combinatorial side'' could be obtained by expanding the symmetric function side in terms of the basis
$\{P_\mu[X;q]\}_\mu$.
\sas

This led to the decision to declare the symmetric function obtained by expanding
the polynomial in (\ref{eq:SymSide}) in terms of the basis $\{P_\mu[X;q]\}_\mu$ as the ``combinatorial side'' of (\ref{eq:SymSide}). This decision led us to conjecture the following  ``combinatorial side'' of
(\ref{eq:polyToStudy}).
\begin{equation}\label{eq:combSide}
RHS_{k,m,n}[X;q]
=q^m\sum_{j=2+k}^{m+1}
q^{{k+2\choose 2}-(k+2)j+1}
\Big[{j-2\atop k}\Big]_q\Big[{m-1\atop j-2}\Big]_q
(q;q)_j\bu\sum_{\mu\part n;\ell(\mu)=j}
\bu q^{n(\mu)}P_\mu[X;q].
\end{equation}
In this paper, we first prove that
\begin{equation}\label{eq:symSideEqualsCombSide}
LHS_{k,m,n}[X;q]\ses  RHS_{k,m,n}[X;q].
\end{equation}
\sas

Jeff Remmel succeeded in formulating many of the conjectures needed to prove (\ref{eq:symSideEqualsCombSide}) by precisely following the methods developed in \cite{GHRY}.
In the first section we will outline the proof of (\ref{eq:symSideEqualsCombSide}) and walk through the steps
used by Jeff Remmel to formulate his conjectures needed to complete this proof.  In the second section, we present the technical details carried out by the remaining authors to prove Remmel's conjectures and ultimately prove (\ref{eq:symSideEqualsCombSide}).
\sas

After this project was completed, we learned that Brendon Rhoades and Mark Shimozono had already constructed, for any partition $\nu$, a symmetric function to be viewed as the ``combinatorial side'' and conjectured it to be equal to the
polynomial
\begin{equation}\label{eq:RSSymSide}
{\it LHS}_{\nu,n}[X,q]\ses \om\Big(\DD_{s_{\nu}}'e_n\Big|_{t=0}
\Big)[X].
\end{equation}
Even more importantly, Jim Haglund communicated to us that he was able to prove the Rhoades-Shimozono conjectures using solely the results in  \cite{GHRY}.  We show here that an appropriate modification of Haglund's argument proves that the polynomial in (\ref{eq:RSSymSide}) plethystically evaluated at $X(1-q)$ expands only in terms of hook Schur functions for all $\nu$.  This confirms our original experimental findings about the polynomial in (\ref{eq:SymSide}).
\sas

These truly surprising circumstances demanded  at least two additional investigations.
The first was to determine whether or not there was any relation between our method of predicting a ``combinatorial side'' and
the Rhoades-Shimozono conjectures.  The second was to find a symmetric function reason explaining Haglund's result.  In the final section of  the paper, we present our comments about  these two problems. Here we will add a few words.
\sas

For the first problem the evidence we gathered confirms that in this case
 our combinatorial side predicts the Rhoades-Shimozono combinatorial side.
\sas

To be precise, we show that the symmetric function
\[
LHS_{\nu,n}[X,q]\ses \om\Big(\DD'_{s_\nu}\, e_n\Big|_{t=0}\Big)[X(1-q)]
\]
expands in terms of the $\{P_\mu[X, q^{-1}]\}_\mu$ basis as
\begin{equation}\label{eq:expSymSidePmu}
LHS_{\nu,n}[X,q]
\ses
q^{|\nu| } \sum_{\mu\part n}
s_\nu\big[\tttt{1-q^{\l(\mu)-1}\over  1-q }\big]
q^{-n(\mu)}
(q;q)_{\l(\mu) }
  P_\mu(X,q^{-1}).
\end{equation}
Expanding the polynomial in (\ref{eq:expSymSidePmu}) in terms of the basis $\{P_\mu[X,q]\}_\mu$ yielded our conjectured ``combinatorial side'' to  be the symmetric function
\begin{align*}
RHS_{\nu,n}[X,q]
=&
q^{|\nu| }
\sum_{k=\l(\nu)}^{|\nu|}(q;q)_k
\sum_{|\rho|=|\nu|, , \,
\l(\rho)=k}
\,\,
{K_{\nu,\rho}(q)
\over
\prod_{i=1}^m
(q;q)_{m_i(\rho)}}
q^{n(\rho)}\,\times
\\
&\bigsp\bigsp\bigsp
\times \,
q^{-k(k+1)}
(q;q)_{k+1}\uu\uu\sum_{\multi{\mu\part n\, ;\, \l(\mu)=k+1}}\uu\uu
q^{n(\mu)}
 P_\mu[X;q]
\end{align*}
It turns out that this is precisely the Rhoades-Shimozono ``combinatorial side'' plethystically evaluated at $X(1-q)$.

\section{Jeff Remmel's conjectures in the hook case}

In this section, we will outline the steps followed by Remmel to formulate the conjectures necessary to establish the equality in (\ref{eq:symSideEqualsCombSide}), that is,
\[
LHS_{k,m,n}[X;q]\ses  RHS_{k,m,n}[X;q]
\]
with the polynomial in (\ref{eq:equivalentPoly}) :
\[
{\it LHS}_{k,m,n}[X,q]= q^{m+{k+1\choose2}}
\sum_{\mu\part  n}   q^{-n(\mu)}(q;q)_{\ell (\mu)}
\Big[{  m -1\atop k}\Big]_q
\Big[{m +\ell(\mu) -(k+2)\atop m}\Big]_q  P_\mu [X;q^{-1}]
\]
as the ``symmetric function  side'', and the polynomial in (\ref{eq:combSide}):
\[
RHS_{k,m,n}[X;q]
=q^m\sum_{j=2+k}^{m+1}
q^{{k+2\choose 2}-(k+2)j+1}
\Big[{j-2\atop k}\Big]_q\Big[{m-1\atop j-2}\Big]_q
(q;q)_j\bu\sum_{\mu\part n;\ell(\mu)=j}
\bu q^{n(\mu)}P_\mu[X;q]
\]
as  the ``combinatorial side''.

To follow the classical case, Remmel used the identity
\begin{equation}\label{eq:CauchyKernelIdentity}
{h_n[X(1-q^i)]\over  1-q^i}\ses\sum_{\mu\part n}
q^{n(\mu)} P_\mu[X;q]
\prod_{j=2}^{\l(\mu)}
(1-q^{i- j+1 }  )
\end{equation}
and then tried to solve for the $c_i^{k,m}(q)$ in the equations
\[
RHS_{k,m,n}[X;q]\ses
\sum_{i=1}^n c_i^{k,m}(q)\,
\sum_ {\mu\part n }q^{n(\mu)}  P_\mu[X;q]
\prod_{r=2}^{\ell(\mu)}
(1-q^{i- r+1 }  ),
\]
which may be best rewritten as
\begin{equation}\label{eq:RHSexp}
RHS_{k,m,n}[X;q]=\sum_ {\mu\part n }q^{n(\mu)}  P_\mu[X;q]
\sum_{i=1}^n c_i^{k,m}(q)\,
\prod_{r=2}^{\ell(\mu)}
(1-q^{i- r+1 }  ).
\end{equation}
Likewise (\ref{eq:combSide}) may also be rewritten as
\begin{equation}\label{eq:RHSaltExp}
RHS_{k,m,n}[X;q]  =
\sum_  {\mu\part n }
q^{n(\mu)} P_\mu[X;q] q^m q^{{k+2\choose 2}-(k+2)\ell(\mu)+1}\Big[{\ell(\mu)-2\atop k}\Big]_q\Big[{m-1\atop \ell(\mu)-2}\Big]_q
(q;q)_j.
\end{equation}
Since $\{P_\mu[X;q]\}_\mu$ is a  basis, the equality of (\ref{eq:RHSexp}) and (\ref{eq:RHSaltExp}) can be true if and only if we have
\begin{equation}\label{eq:RemmelConj1}
\sum_{i=1}^n c_i^{k,m}(q)\,
\prod_{r=2}^{j}
(1-q^{i- r+1 })
=
q^{m+{k+2\choose 2}-(k+2)j+1}\Big[{j-2\atop k}\Big]_q\Big[{m-1\atop j-2}\Big]_q
(q;q)_j.
\end{equation}

A careful examination of computer data led  Jeff Remmel to conjecture that the solution of the equations in (\ref{eq:RemmelConj1}) are the
coefficients
\begin{equation}\label{eq:coeffs}
c_{s}^{k,m}(q)
\ses
(-1)^{m+1-s}q^{{ m+1-s \choose 2}-(k+1)m+{k+1\choose 2}}
\Big[
{ m-1 \atop k
}
\Big]_q
\Big[
{ k+2 \atop  m+1-s
}
\Big]_q (1-q^s)
\end{equation}

It turns  out that the proof of the  Remmel conjecture is an easy consequence of the nature of the equations in (\ref{eq:RemmelConj1}). This gives the validity of (\ref{eq:RemmelConj1}) with the $c_i^{k,m}(q)$ given by (\ref{eq:coeffs}). This also proves  the identity
\[
\sum_{i=1}^n c_{i}^{k,m}(q)
{h_n\big[X(1-q^i)\big]\over 1-q^i
}\ses RHS_{k,m,n}[X;q]
\ess\ess\ess
(\hbox{for all $1\le k\le m-1 $ and
$m<n$})
\]
This given, to prove (\ref{eq:symSideEqualsCombSide}) we only
need to show that we also have
\[
\sum_{i=1}^n c_{i}^{k,m}(q)
{h_n\big[X(1-q^i)\big]\over 1-q^i
}\ses LHS_{k,m,n}[X;q]
\ess\ess\ess
(\hbox{for all $1\le k\le m-1 $ and
$m<n$}).
\]
However here,  as in the classical case, rather than the expression in (\ref{eq:CauchyKernelIdentity}) Remmel was forced to use the equivalent expression
\[
{h_n[X(1-q^i)]\over  1-q^i}\ses
\sum_{\mu\part n}
q^{-n(\mu)} P_\mu[X;1/q]
\prod_{j=2}^{\l(\mu)}
(1-  q^{i+ j-1 }  ).
\]
This given, his next goal was to
 prove the identity
\begin{align*}
&
\sum_{\mu\part n}
q^{-n(\mu)} P_\mu[X;1/q]
\sum_{i=1}^n c_{i}^{k,m}(q)
\prod_{j=2}^{\l(\mu)}
(1-  q^{i+ j-1 }  )\ses
\\
\ses
&
\sum_{\mu\part  n}
 q^{-n(\mu)}P_\mu [X;q^{-1}]
q^{m-k-1+{k+2\choose2}}
\Big[{  m -1\atop k}\Big]_q
\Big[{m +\ell(\mu) -(k+2)\atop m}\Big]_q(q;q)_{\ell (\mu)}.
\end{align*}
Since $\{P_\mu [X;q^{-1}]\}_\mu $ is a symmetric function basis, equating the coefficients of $P_\mu [X;q^{-1}]$ on both sides reduced us to verifying the following $q$-identity for all $1\le \l\le n$
\begin{equation}\label{eq:RemmelConj2}
\sum_{i=1}^n c_{i}^{k,m}(q)
\prod_{j=2}^{\l}
(1-  q^{i+ j-1 }  )=
q^{m+{k+1\choose2}}
\Big[{  m -1\atop k}\Big]_q
\Big[{m +\l -(k+2)\atop m}\Big]_q(q;q)_{\l}.
\end{equation}
Actually, in order to  prove (\ref{eq:symSideEqualsCombSide}), we need only show that by means of the Remmel's coefficients defined in (\ref{eq:coeffs}), both of his conjectures (\ref{eq:RemmelConj1}) and (\ref{eq:RemmelConj2}) hold.  The following section contains all the details needed to carry this out.

\section{Technical details}

In this section, we provide the technical details that are needed to prove the Remmel conjectures.  We begin with a particular $q$-binomial identity.

\begin{prop}\label{prop:mainqidentity}
  Given nonnegative integers $m,k,\ell$ with $k+2\leq \ell\leq m+1$,
  \[ \sum_{i=0}^{\min(k+2,m+1-\ell)}(-1)^i q^{\binom{i}{2}} \qbin{k+2}{i}{q} \qbin{m+1-i}{\ell}{q}=q^{(k+2)(m+1-\ell)} \qbin{m-k-1}{\ell-2-k}{q}.\]
\end{prop}
\begin{proof}
  We will show that the proposition is a consequence of a well known hypergeometric series identity.  First, we put it in standard form.  Let $$t_j=(-1)^j q^{\binom{j}{2}} \qbin{k+2}{j}{q} \qbin{m+1-j}{\ell}{q}.$$  Then, the ratio of consecutive terms in the summation is $\frac{t_{j+1}}{t_j}$ which after some simplification can be shown to be equal to $\frac{-q^{k-\ell+2}(1-q^{-2-k}q^j)(1-q^{\ell-m-1}q^j)}{(1-q^{j+1})(1-q^{-m-1}q^j)}.$  Thus we can write the summation appearing on the left hand side of the proposition as a hypergeometric series, \begin{equation}\label{eq:hgs} \setlength\arraycolsep{1pt}
\qbin{m+1}{\ell}{q}{}_2 \Phi_1\left(\begin{matrix}q^{-2-k}, &q^{\ell-m-1}\\& q^{-m-1}\end{matrix}\bigg\rvert q;q^{k-\ell+2}\right).\end{equation}
 The $q$-Vandermonde hypergeometric series identity asserts that
 $${}_2 \Phi_1\left(\begin{matrix}A, &q^{-n}\\& C\end{matrix}\bigg\rvert q;\frac{C}{Aq^{-n}}\right)=\frac{(\frac{C}{A};q)_n}{(C,q)_n}.$$  Applying this to (\ref{eq:hgs}) yields \begin{equation}\label{eq:evalhgs} \qbin{m+1}{\ell}{q}\frac{(q^{k-m+1};q)_{m-\ell+1}}{(q^{-m-1};q)_{m-\ell+1}}.\end{equation}
 Using the identity
 \begin{equation}(q^{-n};q)_m=q^{m(m-2n-1)/2}(-1)^m(q^{n-m+1};q)_m, \end{equation}
  equation (\ref{eq:evalhgs}) can be simplified to
  \[ q^{(k+2)(m+1-\ell)} \qbin{m+1}{\ell}{q}\frac{(q^{\ell-k-1};q)_{m-\ell+1}}{(q^{\ell+1};q)_{m-\ell+1}}, \] which can easily be manipulated to become the right hand side of the proposition.
\end{proof}
The identity given in Proposition (\ref{prop:mainqidentity}) gives rise to the following corollary under the substitution $m\rightarrow m-1+\ell$.
\begin{cor}\label{cor:mainqidentity}
\[ \sum_{i=0}^{\min(k+2,m)} (-1)^iq^{\binom{i}{2}}\qbin{k+2}{i}{q}\qbin{m+\ell-i}{\ell}{q}=q^{(k+2)m}\qbin{m+\ell-(k+2)}{\ell-(k+2)}{q}. \]
\end{cor}
What follows next is a proposition which completely verifies Remmel's conjectures.  Namely, that given the coefficients defined in (\ref{eq:coeffs}), both (\ref{eq:RemmelConj1}) and (\ref{eq:RemmelConj2}) hold.
\begin{prop}\label{prop:simplifyingproduct}Given nonnegative integers $k,m,n,\ell$ with $k+2\leq \ell \leq m+1\leq n$,
\begin{enumerate}
\item $\displaystyle \sum_{i=1}^n c_i^{k,m}\prod_{j=2}^{\ell}(1-q^{i-j+1})=q^{m+\binom{k+2}{2}-(k+2)\ell+1}\qbin{\ell-2}{k}{q}\qbin{m-1}{\ell-2}{q}(q;q)_\ell$ \label{prop:simpa}
\item $\displaystyle \sum_{i=1}^n c_i^{k,m}\prod_{j=2}^{\ell}(1-q^{i+j-1})=q^{m+\binom{k+1}{2}}\qbin{m-1}{k}{q}\qbin{m+\ell-(k+2)}{m}{q}(q;q)_\ell$ \label{prop:simpb}
\end{enumerate}
    \end{prop}
\begin{proof}
First, it is worth noting that by our definitions $c_i^{k,m}=0$ when either $i>m+1$ or $i<m-k-1$.
To prove part~\ref{prop:simpa}, notice that when $i<\ell$ the product contains a 0   term.  Thus,
\begin{eqnarray*}
&&\sum_{i=1}^n c_i^{k,m}\prod_{j=2}^{\ell}(1-q^{i-j+1})\\
&=&\sum_{i=\max(m-k-1,\ell)}^{m+1} c_i^{k,m}\prod_{j=2}^{\ell}(1-q^{i-j+1})\\
&=&\sum_{i=\max(m-k-1,\ell)}^{m+1} c_i^{k,m}\frac{(1-q)\cdots(1-q^{i-1})}{(1-q)\cdots(1-q^{i-\ell})}\\
&=&\sum_{i=\max(m-k-1,\ell)}^{m+1} c_i^{k,m}\qbin{i}{\ell}{q}\frac{(q;q)_\ell}{1-q^i}\\
&=&\sum_{i=\max(m-k-1,\ell)}^{m+1} (-1)^{m+1-i}q^{\binom{m+1-i}{2}-(k+1)m+\binom{k+1}{2}}\qbin{m-1}{k}{q}\qbin{k+2}{m+1-i}{q}\qbin{i}{\ell}{q}(q;q)_\ell\\
&=&\sum_{i=0}^{\min(k+2,m+1-\ell)} (-1)^{i}q^{\binom{i}{2}-(k+1)m+\binom{k+1}{2}}\qbin{m-1}{k}{q}\qbin{k+2}{i}{q}\qbin{m+1-i}{\ell}{q}(q;q)_\ell\\
&=&q^{\binom{k+1}{2}-m(k+1)}\qbin{m-1}{k}{q}\sum_{i=0}^{\min(k+2,m+1-\ell)} (-1)^{i}q^{\binom{i}{2}}\qbin{k+2}{i}{q}\qbin{m+1-i}{\ell}{q}(q;q)_\ell.\\
\end{eqnarray*}
Then using Proposition \ref{prop:mainqidentity},
\begin{eqnarray*}
&=&q^{\binom{k+1}{2}-m(k+1)}\qbin{m-1}{k}{q} q^{(k+2)(m+1-\ell)} \qbin{m-k-1}{\ell-2-k}{q}(q;q)_\ell\\
&=&q^{m+\binom{k+2}{2}-(k+2)\ell+1} \qbin{\ell-2}{k}{q}\qbin{m-1}{\ell-2}{q}(q;q)_\ell.\\
\end{eqnarray*}
This completes the proof of part~\ref{prop:simpa}.  To prove part~\ref{prop:simpb},
\begin{eqnarray*}
&&\sum_{i=1}^n c_i^{k,m}\prod_{j=2}^{\ell}(1-q^{i+j-1})\\
&=&\sum_{i=\max(m-k-1,1)}^{m+1} c_i^{k,m}\prod_{j=2}^{\ell}(1-q^{i+j-1})\\
&=&\sum_{i=0}^{\min(k+2,m)} c_{m+1-i}^{k,m}\prod_{j=2}^{\ell}(1-q^{m+j-i})\\
&=&q^{\binom{k+1}{2}-(k+1)m}\qbin{m-1}{k}{q}\sum_{i=0}^{\min(k+2,m)} (-1)^iq^{\binom{i}{2}}\qbin{k+2}{i}{q}(1-q^{m+1-i})\prod_{j=2}^{\ell}(1-q^{m+j-i})\\
&=&q^{\binom{k+1}{2}-(k+1)m}\qbin{m-1}{k}{q}\sum_{i=0}^{\min(k+2,m)} (-1)^iq^{\binom{i}{2}}\qbin{k+2}{i}{q}\qbin{m+\ell-i}{\ell}{q}(q;q)_\ell\\
&=&q^{\binom{k+1}{2}-(k+1)m}\qbin{m-1}{k}{q}q^{(k+2)m}\qbin{m+\ell-(k+2)}{\ell-(k+2)}{q}(q;q)_\ell\\
&=&q^{m+\binom{k+1}{2}}\qbin{m-1}{k}{q}\qbin{m+\ell-(k+2)}{m}{q}(q;q)_\ell
\end{eqnarray*}
The next to last step is justified by Corollary \ref{cor:mainqidentity}.
\end{proof}

\section{Additional investigations}
To begin the investigation of whether our ``combinatorial side'' was related to that of Rhodes-Shimozono, we first   expanded the symmetric function $\om\Big(\DD_{s_{\nu}}'e_n\Big|_{t=0}\Big)[X(1-q)]$ in terms of the basis $\{P_\mu(X,q^{-1})\}_\mu$.  This is done in the following theorem.
\begin{thm}\label{thm:genSymSideExp}
\begin{equation}\label{eq:thmGenSymSideExp}
\om\Big(\DD_{s_{\nu}}'e_n\Big|_{t=0}
\Big)[X(1-q)]
 \ses q^{|\nu| } \sum_{\mu\part n}
s_\nu\big[\tttt{1-q^{\l(\mu)-1}\over  1-q }\big]
q^{-n(\mu)}
(q;q)_{\l(\mu) }
  P_\mu(X,q^{-1})
\end{equation}
\end{thm}
\begin{proof}
We begin with the following expansion of $e_n$ (Lemma 2.1 in \cite{GHRY}),
\[ e_n (X) = \sum_{\mu\vdash n}\frac{(1-q)(1-t)\widetilde{H}_\mu (X;q,t) \Pi_\mu ^{\prime}(q,t)B_\mu (q,t) }{w_\mu (q,t)}.
\]
Recognizing that the left hand side does not contain the indeterminates $q$ and $t$, we can interchange them and obtain
\[ e_n (X) = \sum_{\mu\vdash n}\frac{(1-q)(1-t)\widetilde{H}_\mu (X;t,q) \Pi_\mu ^{\prime}(t,q)B_\mu (t,q) }{w_\mu (t,q)}.
\]
Then using the definition of $\DD'$, and setting $t=0$, we have
\begin{equation}\label{eq:expOfEn}\DD_{s_{\nu}}'e_n\Big|_{t=0} = \sum_{\mu\vdash n}\frac{(1-q)s_{\nu}\big[B_\mu(0,q) -1 \big]\widetilde{H}_\mu (X;0,q) \Pi_\mu ^{\prime}(0,q)B_\mu (0,q) }{w_\mu (0,q)}.
\end{equation}
In \cite{GHRY}, it was noted that
\begin{align*}
B_\mu (0,q) &= 1+q+\dotsb +q^{\ell (\mu)-1}=\frac{1-q^{\ell(\mu)}}{1-q},\\
 \Pi_\mu ' (0,q) &= (q;q)_{\ell (\mu)-1},\\
w_\mu (0,q) &= \prod_{c\in \mu}q^{l(c)} \cdot \prod_{\substack{c\in \mu\\ a(c)=0}}(1-q^{l(c)+1})\cdot \prod_{\substack{c\in \mu \\ a(c)>0}}(-q^{l(c)+1})\\
&= (-1)^{n-\ell (\mu)}q^{2n(\mu)+n-\sum_i \binom{m_i(\mu) +1}{2}}\prod_{i}(q;q)_{m_i(\mu)},
\end{align*}
where $(q;q)_m = (1-q)\cdots (1-q^m)$.
Substituting these into (\ref{eq:expOfEn}) and simplifying gives
\[
\DD_{s_{\nu}}'e_n\Big|_{t=0} = \sum_{\mu\vdash n}(-1)^{n-\ell (\mu)}s_{\nu}\big[q+q^2+\dotsb +q^{\ell (\mu)-1} \big]q^{-2n(\mu)-n+\sum_i \binom{m_i(\mu) +1}{2}}\Big[{ \l(\mu) \atop m(\mu) }\Big]_q \widetilde{H}_\mu (X;0,q).
\]
Replacing $X$ by $X(1-q)$ and factoring a $q$ out of the plethystic evaluation, the right hand side becomes
\[
q^{|\nu|}\sum_{\mu\vdash n}(-1)^{n-\ell (\mu)}s_{\nu}\Big[\frac{1-q^{\ell (\mu)-1}}{1-q} \Big]q^{-2n(\mu)-n+\sum_i \binom{m_i(\mu) +1}{2}}\Big[{ \l(\mu) \atop m(\mu) }\Big]_q \widetilde{H}_\mu (X(1-q);0,q),
\]
and then expanding $\widetilde{H}_\mu (X(1-q);0,q)$ yields
\[
q^{|\nu|}\sum_{\mu\vdash n}(-1)^{n-\ell (\mu)}s_{\nu}\Big[\frac{1-q^{\ell (\mu)-1}}{1-q} \Big]q^{-2n(\mu)-n+\sum_i \binom{m_i(\mu) +1}{2}}\Big[{ \l(\mu) \atop m(\mu) }\Big]_q \sum_{\la\part n}s_\la\left[X(1-q)\right]\widetilde{K}_{\lambda , \mu}(q).
\]
But, since $s_\la\left[X(1-q)\right]=(-q)^ns_{\la'}\left[X(1-1/q)\right]$ and $\widetilde{K}_{\lambda , \mu}(q) = q^{n(\mu)}K_{\lambda,\mu}(q^{-1})$, we can now apply $\om$ and eventually arrive at
\[
q^{|\nu|}\sum_{\mu\vdash n}(-1)^{\ell (\mu)}s_{\nu}\Big[\frac{1-q^{\ell (\mu)-1}}{1-q} \Big]q^{-n(\mu)+\sum_i \binom{m_i(\mu) +1}{2}}\Big[{ \l(\mu) \atop m(\mu) }\Big]_q \sum_{\la\part n}s_{\la}\left[X(1-q^{-1})\right]K_{\lambda , \mu}(q^{-1}).
\]
We will next need  two  facts stated in \cite{GHRY}:
\begin{equation}\label{eq:fact1fromGHRY}
 Q_\mu(X,q) =  \sum_{\la \part n}s_\la\left[X(1-q)\right]K_{\la,\mu}(q)
\end{equation}
and
\begin{equation}\label{eq:fact2fromGHRY}
 P_\mu(X,q^{-1}) =  \frac{(-1)^{\ell(\mu)} q^{\sum_{i}{m_i(\mu)+1\choose 2}}}{\prod (q;q)_{m_i(\mu)}}Q_\mu(X,q^{-1}),
\end{equation}
where $\mu$ is a partition of $n$.
Applying (\ref{eq:fact1fromGHRY}) at $q^{-1}$ gives
\[
q^{|\nu|}\sum_{\mu\vdash n}(-1)^{\ell (\mu)}s_{\nu}\Big[\frac{1-q^{\ell (\mu)-1}}{1-q} \Big]q^{-n(\mu)+\sum_i \binom{m_i(\mu) +1}{2}}
{(q;q)_{\l(\mu)}
\over
\prod_{i=1}^{\l(\mu)}(q;q)_{m_i(\mu)}
}\,\, Q_\mu(X,q^{-1}),
\]
and then applying (\ref{eq:fact2fromGHRY}) we prove the theorem, namely,
$$
\om\Big(\DD_{s_{\nu}}'e_n\Big|_{t=0}
\Big)[X(1-q)]
 \ses q^{|\nu| } \sum_{\mu\part n}
s_\nu\big[\tttt{1-q^{\l(\mu)-1}\over  1-q }\big]
q^{-n(\mu)}
(q;q)_{\l(\mu) }
  P_\mu(X,q^{-1}).
$$

\end{proof}

\begin{cor}
The identity (\ref{eq:equivalentPoly}), namely,
\[
{\it LHS}_{k,m,n}[X,q]=
q^{m +{k+1\choose2}}
\sum_ {\mu\part  n}
\Big[{  m -1\atop k}\Big]_q
\Big[{m +\ell(\mu) -(k+2)\atop m}\Big]_q q^{-n(\mu)}(q;q)_{\ell (\mu)} P_\mu [X;q^{-1}],
\]
is none other but a specialization of Theorem \ref{thm:genSymSideExp}, at $\nu=(m-k,1^k)$.
\end{cor}

\vskip  -.12in
\begin{flushright}
\begin{tikzpicture}[scale=0.46, line width=1pt]
  \draw (0,0) grid (1,-5);
  \draw (1,-4) grid (4,-5);
  \node at (.5,-.5) {-4};
  \node at (.5,-1.5) {-3};
  \node at (.5,-2.5) {-2};
  \node at (.5,-3.5) {-1};
  \node at (.5,-4.5) {0};
  \node at (1.5,-4.5) {1};
  \node at (2.5,-4.5) {2};
  \node at (3.5,-4.5) {3};
  \node[right] at (2,-.5) {$m$-$k$-$1$};
  \draw[->](2,-.5) -- (1,-.5);
  \node[left] at (2.75,-3) {$k$};
  \draw[->](2.5,-3.25) -- (3.5,-4);
  \node at (2,-5.5) {$c(x)$};

  \draw (6,0) grid (7,-5);
  \draw (6,-4) grid (10,-5);
  \node at (6.5,-.5) {1};
  \node at (6.5,-1.5) {2};
  \node at (6.5,-2.5) {3};
  \node at (6.5,-3.5) {4};
  \node at (6.5,-4.5) {8};
  \node at (7.5,-4.5) {3};
  \node at (8.5,-4.5) {2};
  \node at (9.5,-4.5) {1};
  \node[right] at (7.75,-2.25) {$m$-$k$-$1$};
  \draw[->](8,-2.5) -- (6.75,-3.5);
  \node[right] at (7.85,-3.5) {$m$};
  \draw[->](8,-3.5) -- (6.75,-4.5);
  \node at (8,-5.5) {$h(x)$};
\end{tikzpicture}
\end{flushright}

\vskip  -1.6in

\vskip .3in
\begin{proof}

Recall that the definition of the left hand side of (\ref{eq:equivalentPoly}) is
\vskip .1in

\ess\ess\ess$
{\it LHS}_{k,m,n}[X,q]= \om\Big(\DD_{s_{m-k,1^k}}'e_n\Big|_{t=0} \Big)[X(1-q)]
$

\vskip .1in
\noindent
Now the Macdonald formula for the plethystic evaluation of $s_\la$

\vskip .1in
\noindent
at $ 1+q+\cdots+q^{n-1}$  is

\vskip -.12in
$$
s_\la[1+q+\cdots+q^{n-1}]
= q^{n(\la)}\Big[{n \atop  \la'} \Big]_q
$$
\vskip -.16in
\noindent
where
$$
\Big[{n \atop  \la } \Big]_q
\ses \prod_{x\in \la}{  1-q^{n-c(x)} \over 1-q^{h(x)}}
$$
With $c(x)$ and $h(x)$ the {\it content} and the {\it hook} of cell $x\in \la$.
\sas

Now for $\la=(m-k,1^k)$ we have  $n(\la)={k+1\choose 2}$ and
 $\la'=(k+1,1^{m-k-1})$. We thus obtain

\begin{align*}
s_{m-k,1^k}[1+q+\cdots+q^{\ell-2}]&=
q^ {k+1\choose 2}
{
(1-q^{\ell-1+0 })\cdots (1-q^{\ell-1 +m-k-1})(1-q^{\ell-1-1})\cdots (1-q^{\ell-1-k})
\over
(q;q)_k(1-q^m)(q;q)_{m-k-1}
}\\
&=
q^ {k+1\choose 2}
{ (q^{\l-k-1},q)_m
\over
(q;q)_k(1-q^m)(q;q)_{m-k-1}
}
\end{align*}
See the illustration above where  the statistics $c(x)$ and $h(x)$ are computed for the hook partition $(m-k,1^k)$
\sa
\noindent
Notice next that we have
\begin{align*}
\Big[{  m -1\atop k}\Big]_q
\Big[{m +\ell(\mu) -(k+2)\atop m}\Big]_q &\ses
{1
\over
(q;q)_{k}(q;q)_{m-1-k}
}
\ess
{(q;q)_{m +\ell(\mu) -(k+2)}
\over
(1-q^m)(q;q)_{\ell(\mu) -(k+2)}}
\\
&\ses
{(q^{\l-k-1},q)_m
\over
(q;q)_{k}(1-q^m)(q;q)_{m-1-k}
}
\end{align*}

\noindent
To prove that for $\nu=(m-k,1^k)$
(\ref{eq:thmGenSymSideExp}) reduces to (\ref{eq:equivalentPoly}), we need only verify the equality
\[
q^{m }
s_{m-k,1^k}\big[\tttt{1-q^{\l -1}\over  1-q }\big]\ses q^{m +{k+1\choose2}}
\Big[{  m -1\atop k}\Big]_q
\Big[{m +\ell  -(k+2)\atop m}\Big]_q
\]
However, the above calculations show exactly that.
\end{proof}

Theorem \ref{thm:genSymSideExp} provides an expansion of the symmetric function side in terms of the\\ $\{P_\mu(X,q^{-1})\}$ basis.  We now seek an appropriate ``combinatorial side'' by expanding the same symmetric function in terms of the $\{P_\mu(X,q)\}$ basis.  In order to do this, we will use a special evaluation given in the following theorem.

\begin{thm}\label{thm:specEval}
\begin{equation}\label{eq:thmSpecEval}
s_\nu\Big[\tttt{1-q^{j-1}\over  1-q }\Big]
\ses
\sum_{k=\l(\nu)}^{|\nu|}
\sum_{\multi{|\rho|=|\nu|\cr
\l(\rho)=k}}
\,\,
{K_{\nu,\rho}(q)
\over
\prod_{i=1}^m
(q;q)_{m_i(\rho)}}
q^{n(\rho)}
{(q;q)_{j-1}
\over (q;q)_{j-1-k}}
\end{equation}
\end{thm}
\begin{proof}
Recall that from \cite{Mac}, we get the identity
\[
s_\nu[X]\ses \sum_{\rho\part |\nu|}K_{\nu,\rho}(q)P_\rho[X,q],
\]
which can be also written as
\[
s_\nu[X]\ses \sum_{\rho\part |\nu|}K_{\nu,\rho}(q){Q_\rho[X,q]\over \prod_{i=1}^m
(q;q)_{m_i(\rho)}}
\]
and $X\RA X/(1-q)$ gives
\[
s_\nu\big[\tttt{X\over 1-q}\big]\ses \sum_{\rho\part|\nu|}H_\rho[X ;q] { K_{\nu,\rho}(q)\over
\prod_{i=1}^m
(q;q)_{m_i(\rho)}}.
\]
Now the replacement $X\RA 1-q^{j-1}$
yields
\[
s_\nu\big[\tttt{ 1-q^{j-1}\over 1-q}\big]\ses \sum_{\rho\part|\nu|}H_\rho[ 1-q^{j-1} ;q] { K_{\nu,\rho}(q)\over
\prod_{i=1}^m
(q;q)_{m_i(\rho)}}
\]
This can be rewritten in the form
\begin{equation}\label{eq:thmSpecEvalRew}
s_\nu\big[\tttt{1-q^{j-1}\over  1-q }\big]
\ses
\sum_{k=1}^{|\nu|}
\sum_{\multi{|\rho|=|\nu|\cr
\l(\rho)=k}}
\,\,H_\rho[1-q^{j-1};q ]
{K_{\nu,\rho}(q)
\over
\prod_{i=1}^m
(q;q)_{m_i(\rho)}}.
\end{equation}
Now, the Macdonald reciprocity in the Hall-Littlehood case yields
\[
H_\rho[1-u;q ]\ses q^{n(\rho)}
\prod_{s=1}^{\l(\rho)}
(1-u/q^{ s-1 }  ).
\]
In particular, the replacement $u\RA q^{j-1}$ gives (for  $\l(\rho)=k$)
\[
H_\rho[1-q^{j-1};q ]\ses q^{n(\rho)}
\prod_{s=1}^{k}
(1- q^{j -s }  )\ses q^{n(\rho)}(1-q^{j-k})\cdots (1-q^{j-1})
\]
Thus (\ref{eq:thmSpecEvalRew}) becomes
\[
s_\nu\big[\tttt{1-q^{j-1}\over  1-q }\big]
\ses
\sum_{k=1}^{|\nu|}
\sum_{\multi{|\rho|=|\nu|\cr
\l(\rho)=k}}
\,\,
{K_{\nu,\rho}(q)
\over
\prod_{i=1}^m
(q;q)_{m_i(\rho)}}
q^{n(\rho)}
{(q;q)_{j-1}
\over (q;q)_{j-1-k}}
\]
Since the coefficient $K_{\nu,\rho}(q)$ fails to vanish only when  $\nu\ge\rho$ in dominance,
the hypothesis $\l(\rho)=k$ forces $\l(\nu)\le k$. This given we can write
\[
s_\nu\big[\tttt{1-q^{j-1}\over  1-q }\big]
\ses
\sum_{k=\l(\nu)}^{|\nu|}
\sum_{\multi{|\rho|=|\nu|\cr
\l(\rho)=k}}
\,\,
{K_{\nu,\rho}(q)
\over
\prod_{i=1}^m
(q;q)_{m_i(\rho)}}
q^{n(\rho)}
{(q;q)_{j-1}
\over (q;q)_{j-1-k}}
\]
\end{proof}

\vskip -.35in
Now Theorem \ref{thm:genSymSideExp} gives that our symmetric function side has the expansion
\[
LHS_{\nu,n}[X,q]
\ses
q^{|\nu| } \sum_{\mu\part n}
s_\nu\big[\tttt{1-q^{\l(\mu)-1}\over  1-q }\big]
q^{-n(\mu)}
(q;q)_{\l(\mu) }
  P_\mu(X,q^{-1}).
\]
Using Theorem \ref{thm:specEval}, this can be rewritten as
\[
q^{|\nu| } \sum_{\mu\part n}
\sum_{k=\l(\nu)}^{|\nu|}
\sum_{\multi{|\rho|=|\nu|\cr
\l(\rho)=k}}
\,\,
{K_{\nu,\rho}(q)
\over
\prod_{i=1}^m
(q;q)_{m_i(\rho)}}
q^{n(\rho)}
{(q;q)_{\l(\mu)-1}
\over (q;q)_{\l(\mu)-1-k}}
q^{-n(\mu)}
(q;q)_{\l(\mu) }
  P_\mu(X,q^{-1}),
\]
or better,
\[
q^{|\nu| }
\sum_{k=\l(\nu)}^{|\nu|}(q;q)_k
\sum_{\multi{|\rho|=|\nu\cr
\l(\rho)=k}}
\,\,
{K_{\nu,\rho}(q)
\over
\prod_{i=1}^m
(q;q)_{m_i(\rho)}}
q^{n(\rho)}\,
\sum_{\mu\part n}\Big[{\l(\mu)-1\atop k}\Big]_q
q^{-n(\mu)}
(q;q)_{\l(\mu) }
  P_\mu(X,q^{-1}).
\]

Recall that in \cite{GHRY}, for the classical case of the Delta conjecture at $t=0$, we proved the identity
\[
\sum_{\mu\part n}q^{- n(\mu)}
P_\mu[X;q^{-1}]
\Big[{ \l(\mu)-1\atop k }\Big]_q
(q;q)_{\l(\mu)}
\ses q^{-k(k+1)}
(q;q)_{k+1}\uu\uu\sum_{\multi{\mu\part n\, ;\, \l(\mu)=k+1}}\uu\uu
q^{n(\mu)}
 P_\mu[X;q]
\]
This permits us to obtain
the expansion of the symmetric function side in terms of the basis
$\big\{P_\mu[X;q]\big\}_\mu$ and use our recipe to obtain what we would label as the ``combinatorial side''.  Namely,
\begin{align*}
RHS_{\nu,n}[X,q]
\ses&
q^{|\nu| }
\sum_{k=\l(\nu)}^{|\nu|}(q;q)_k
\sum_{\multi{|\rho|=|\nu\cr
\l(\rho)=k}}
\,\,
{K_{\nu,\rho}(q)
\over
\prod_{i=1}^m
(q;q)_{m_i(\rho)}}
q^{n(\rho)}\,\times
\\
&\bigsp\bigsp
\times \,
q^{-k(k+1)}
(q;q)_{k+1}\uu\uu\sum_{\multi{\mu\part n\, ;\, \l(\mu)=k+1}}\uu\uu
q^{n(\mu)}
 P_\mu[X;q]
\end{align*}

Additionally, the last sum appearing in ${\it RHS}_{\nu,n}[X,q]$ was proved to have a hook Schur function expansion in \cite{GHRY}.
We have thus proved the following generalization of (\ref{eq:symSideEqualsCombSide}).
\begin{thm}It is not only true that
\begin{equation}\label{eq:thmGeneralSymSideEqualsCombSide}
{\it LHS}_{\nu,n}[X,q]\ses {\it RHS}_{\nu,n}[X,q],
\end{equation}
but also that the Schur expansion of both sides contains only hook
Schur functions.
\end{thm}

\begin{remark}
The right hand side of this identity is none one other than the Rhodes-Shimozono ``combinatorial side'' transformed to our set up, (see the righthand side of  Theorem 1.2 in \cite{HagRhoShi}).
\end{remark}

\begin{remark}
In \cite{GHRY} (see Lemma 4.2) it is shown that
$$
h_n[X(1-u)] \ses (1-u)\sum_{s=0}^{n-1}(-u)^s
s_{(n-s,1^s)}[X]
$$
It follows from this identity that any symmetric polynomial whose Schur functions expansion contains only hook Schur functions may be expanded as  linear combination of the shifted Cauchy  kernel $h_n(X(1-q^i)]$. What forced Remmel to restrict himself to $\Delta_{s_\nu} e_n$ in the hook case of $\nu$ is that in the hook case the needed coefficients are products of $q$-analogues of integers. This facilitated conjecturing their exact nature. With the wisdom of hindsight we can now explain Haglund result as due to the fact that Schur function expansions of  the appropriately modified polynomials
$\Delta_{s_\nu}e_n$ contain only hook Schur functions in full generality. However, this
circumstance is only an artifact of the specialization at $t=0$. In fact, without this specialization, computer data reveals the dimension of the space spanned by the polynomials $\Delta_{s_\nu}e_n$ to be much larger than $n$.
The data suggests that,  more likely,  this dimension is the number of partitions of $n$.
\end{remark}

\section{Acknowledgements}

The authors want to  express their gratitude  to Jim Haglund for making his result  and his argument available to us, (unpublished manuscript). We are also grateful to Dennis Stanton for providing us with the tools we needed to be able to prove (in section 3) Remmel's $q$-binomial conjectures.

\end{document}